\documentclass[10pt]{article}

\DeclareMathAlphabet{\mathpzc}{OT1}{pzc}{m}{it}

\usepackage{amsmath}
\usepackage{amssymb}
\usepackage{amsfonts}
\usepackage{amsthm}
\usepackage{mathrsfs}
\usepackage{enumerate}
\usepackage[pdftex]{color, graphicx}

\allowdisplaybreaks

\newtheorem{thm}{Theorem}[section]

\newtheorem{prop}[thm]{Proposition}
\newtheorem{cor}[thm]{Corollary}

\begin{document}

\renewcommand{\thefootnote}{\arabic{footnote}}
 	
\title{Real elliptic curves and cevian geometry}

\author{\renewcommand{\thefootnote}{\arabic{footnote}}
Igor Minevich and Patrick Morton}
\maketitle

\begin{section}{Introduction.}

In this paper we will investigate the connection between elliptic curves $E$ defined over the real numbers $\mathbb{R}$ and the cevian geometry which we have worked out in the series of papers \cite{mm1}--\cite{mms}. \medskip

We will rederive some of the facts relating to barycentric coordinates that we discussed in the unpublished papers \cite{mm0} and \cite{mo}.  \medskip

First, some notation.  We let $ABC$ be an ordinary triangle in the extended plane, and $P = (x,y,z)$ be a point not on the sides of $ABC$ or its anticomplementary triangle $K^{-1}(ABC)$, whose homogeneous barycentric coordinates with respect to $ABC$ are $(x,y,z)$.  We note that the isotomic map $\iota$ for $ABC$ has the representation
$$P'=\iota(P) = \iota(x,y,z) = \left(\frac{1}{x}, \frac{1}{y}, \frac{1}{z}\right) = (yz, xz, xy), \ \ xyz \neq 0,$$
and the complement mapping $K$ and its inverse $K^{-1}$ with respect to $ABC$ have the matrix representations
$$K=
\left(
\begin{array}{ccc}
 0 & 1  & 1  \\
 1 & 0  & 1  \\
 1 & 1  &  0 
\end{array}
\right), \ \ \
K^{-1}=\left(
\begin{array}{ccc}
 -1 & 1  & 1  \\
 1 & -1  & 1  \\
 1 & 1  & -1 
\end{array}
\right).
$$
It follows easily that the equations of the sides of the anticomplementary triangle $K^{-1}(ABC)$ are given by
$$K^{-1}(BC): y+z=0, \quad K^{-1}(AC): x+z=0, \quad K^{-1}(AB): x+y=0.$$

The isotomcomplement $Q$ of $P$ with respect to $ABC$ is the point $Q=K(\iota(P))=K(P')$, whose barycentric coordinates, in terms of $x,y,z$, are
$$Q = K(yz,xz,xy)^t = (x(y+z), y(x+z), z(x+y))=(x',y',z').$$ 
The corresponding point $Q'=K(\iota(P')) = K(P)$ for $P'$ has coordinates
$$Q'=(y+z,x+z,x+y).$$
Since the equations of the sides $BC, CA, AB$ of $ABC$ are, respectively, $x=0, y=0$, and $z=0$, the points
$$D = (0,y,z), \ \ E = (x,0,z), \ \ F = (x,y,0)$$ 
are the traces of $P$ on the respective sides, and the unique affine mapping $T_P$ taking $ABC$ to $DEF$ is given in terms of barycentric coordinates by the matrix
\begin{equation}
T_P=
\left(
\begin{array}{ccc}
  0 & x'(x+y)  & x'(x+z)  \\
  y'(x+y) & 0  & y'(y+z)  \\
  z'(x+z) & z'(y+z)  & 0  
\end{array}
\right).
\end{equation}
The traces of the point $P'$ on the sides of $ABC$ are
$$D_3 = (0,z,y), \ \ E_3 = (z,0,x), \ \ F_3 = (y,x,0),$$
and the corresponding affine map $T_{P'}:ABC \rightarrow D_3E_3F_3$ is given by the matrix
\begin{equation}
T_{P'} =
\left(
\begin{array}{ccc}
 0 & z'(y+z)  & y'(y+z)  \\
  z'(x+z) & 0  & x'(x+z)  \\
  y'(x+y) & x'(x+y)  & 0  
\end{array}
\right).
\end{equation}
In \cite{mm1} we showed that $Q$ is a fixed point of $T_P$ and $Q'$ is a fixed point of $T_{P'}$; further, these are the only ordinary fixed points of these maps, when $P$ and $P'$ are ordinary points.  (See \cite{mm1}, Theorems 3.2 and 3.12.) \medskip

In \cite{mm1} and \cite{mm2} we studied the affine mapping
\begin{equation}
\mathcal{S}=T_P \circ T_{P'} =
\left(
\begin{array}{ccc}
  x(y'+z') & xx' & xx' \\
  yy' & y(x'+z')  & yy'  \\
  zz' & zz'  & z(x'+y')  
\end{array}
\right),
\end{equation}
which is a homothety or translation.  In \cite{mm2} our main focus was on the affine map
\begin{equation}
\lambda=T_{P'} \circ T_P^{-1} =
\left(
\begin{array}{ccc}
  yz(y+z) & xz(y-z) & xy(z-y) \\
  yz(x-z) & xz(x+z)  & xy(z-x) \\
  yz(x-y) & xz(y-x)  & xy(x+y)  
\end{array}
\right).
\end{equation}
Then in \cite{mm3} we made use of the map
\begin{equation}
\textsf{M}=T_{P} \circ K^{-1} \circ T_{P'} =
\left(
\begin{array}{ccc}
  x(y-z)^2 & x(y+z)^2 & x(y+z)^2 \\
  y(x+z)^2 & y(x-z)^2  & y(x+z)^2 \\
  z(x+y)^2 & z(x+y)^2  & z(x-y)^2  
\end{array}
\right),
\end{equation}
which is also a homothety or translation.  In \cite{mmv} we studied the points $P$ for which $\textsf{M}$ is a translation, and in \cite{mms} we studied the points $P$ for which $\textsf{M}$ is a half-turn.

\begin{prop}
\label{prop:fixed}
The maps $\mathcal{S}, \lambda$, and $\textsf{M}$ have the respective fixed points
\begin{align}
X &= (xx', yy', zz') = (x^2(y+z),y^2(x+z),z^2(x+y)) = P \cdot Q,\\
Z &= (x(y-z)^2, y(z-x)^2, z(x-y)^2),\\
S &= (x(y+z)^2, y(x+z)^2, z(x+y)^2) = Q \cdot Q'.
\end{align}
\end{prop}

\begin{proof}
This is a straightforward calculation.  For example,
\begin{align*}
\mathcal{S}(X) &= \left(
\begin{array}{ccc}
  x(y'+z') & xx' & xx' \\
  yy' & y(x'+z')  & yy'  \\
  zz' & zz'  & z(x'+y')  
\end{array}
\right)
(xx', yy', zz')^t\\
&= (x+y)(x+z)(y+z) (xx',yy',zz')^t.
\end{align*}
Since $P$ is not on any of the sides of $K^{-1}(ABC)$, the quantity $(x+y)(x+z)(y+z)$ is nonzero, so $\mathcal{S}(X)=X$.  Simlarly,
\begin{align*}
\lambda(Z)&= 
\left(
\begin{array}{ccc}
  yz(y+z) & xz(y-z) & xy(z-y) \\
  yz(x-z) & xz(x+z)  & xy(z-x) \\
  yz(x-y) & xz(y-x)  & xy(x+y)  
\end{array}
\right)(x(y-z)^2,y(z-x)^2,z(x-y)^2)^t \\
&= (2xyz)(x(y-z)^2,y(z-x)^2,z(x-y)^2)^t,
\end{align*}
where $xyz \neq 0$, since $P$ does not lie on the sides of $ABC$; and
\begin{align*}
\textsf{M}(S) &= \left(
\begin{array}{ccc}
  x(y-z)^2 & x(y+z)^2 & x(y+z)^2 \\
  y(x+z)^2 & y(x-z)^2  & y(x+z)^2 \\
  z(x+y)^2 & z(x+y)^2  & z(x-y)^2  
\end{array}
\right)(x(y+z)^2,y(x+z)^2,z(x+y)^2)^t\\
&=\rho(x(y+z)^2,y(x+z)^2,z(x+y)^2)^t,
\end{align*}
where
$$\rho=x(y^2+z^2)+y(x^2+z^2)+z(x^2+y^2)+2xyz=(x+y)(x+z)(y+z).$$
This proves the proposition.
\end{proof}

The points $X$ and $S$ are the centers of the respective maps $\mathcal{S}$ and $\textsf{M}$.  In the former case, letting $Y=(a,b,c)$, we have
\begin{align*}
\mathcal{S}(Y) &= \left(
\begin{array}{ccc}
  x(y'+z') & xx' & xx' \\
  yy' & y(x'+z')  & yy'  \\
  zz' & zz'  & z(x'+y')  
\end{array}
\right)
(a,b,c)^t\\
&= (ax(y'+z')+(b+c)xx',by(x'+z')+(a+c)yy',cz(x'+y')+(a+b)zz')^t\\
&= ((a+b+c)xx'+2axyz,(a+b+c)yy'+2bxyz,(a+b+c)zz'+2cxyz)^t\\
&=(a+b+c)X+2xyzY.
\end{align*}
It follows that $X$ is collinear with $Y$ and $\mathcal{S}(Y)$, for any ordinary point $Y=(a,b,c)$ (because $a+b+c \neq 0$).  This proves the claim that $X$ is the center of $\mathcal{S}$.  The computation
\begin{align*}
\textsf{M}(Y) &= \left(
\begin{array}{ccc}
  x(y-z)^2 & x(y+z)^2 & x(y+z)^2 \\
  y(x+z)^2 & y(x-z)^2  & y(x+z)^2 \\
  z(x+y)^2 & z(x+y)^2  & z(x-y)^2  
\end{array}
\right)(a,b,c)^t\\
&= (a+b+c)S-4xyzY
\end{align*}
shows the same for $S$ and the map $\textsf{M}$.  This verifies that the points $X$ and $S$ are the same as the points (with the same names) which are discussed in \cite{mm1} and \cite{mm3}.  Note that these relations also show that $\mathcal{S}$ and $\textsf{M}$ fix all the points on the line at infinity.  \medskip

Next we define the point
\begin{equation}
\label{eqn:barV}
V=(x(y^2+yz+z^2),y(x^2+xz+z^2),z(x^2+xy+y^2))
\end{equation}
and note that $V$ can also be given by
\begin{equation}
\label{eqn:V}
V=PQ \cdot P'Q', \ \ \ \ (P, P' \ \textrm{ordinary}, P \ \textrm{not on a median}).
\end{equation}
To see this, we note first that $V$ is collinear with $P$ and $Q$, which is immediate from the equation
$$V=-(xy+xz+yz)P+(x+y+z)Q, \ \ \ \ (P, P' \ \textrm{ordinary}).$$
We convert to absolute barycentric coordinates by dividing each point in this equation by the sum of its coordinates.  Since the sum of the coordinates of $V$ is
\begin{align*}
x(y^2+yz+z^2) &+ y(x^2+xz+z^2)+z(x^2+xy+y^2)\\
 &= x^2(y+z)+y^2(x+z)+z^2(x+y)+3xyz\\
&= (x+y+z)(xy+xz+yz),
\end{align*}
then denoting the last expression by $F(0)$ (see the proof of Proposition \ref{prop:ZGV} below) we have the relation
$$\frac{1}{F(0)}V=-\frac{1}{x+y+z}P+\frac{2}{2(xy+xz+yz)}Q.$$
It follows from this that $Q$ is the midpoint of the segment $PV$, when $P$ and $V$ are ordinary.  Replacing $P$ by $P'$ and $Q$ by $Q'$ is effected by the map $x \rightarrow \frac{1}{x}, y \rightarrow \frac{1}{y}, z \rightarrow \frac{1}{z}$, and on multiplying through by $x^2y^2z^2$, we obtain the equation
$$V=-(x+y+z)P'+(xy+xz+yz)Q'.$$
This shows as above that $Q'$ is the midpoint of $PV'$.
This proves that (\ref{eqn:V}) holds for the point defined by (\ref{eqn:barV}). \medskip

\begin{prop}
\label{prop:ZGV}
The point $Z$ is collinear with $G=(1,1,1)$ and $V=(x(y^2+yz+z^2),y(x^2+xz+z^2),z(x^2+xy+y^2))$.  We have the relation
\begin{equation}
Z=(-3xyz)G+V.
\label{eqn:Z}
\end{equation}
If $P, P',$ and $Z$ are ordinary points, then we have the signed ratio
\begin{equation}
\frac{GZ}{ZV}=\frac{-1}{9}\frac{(x+y+z)(xy+yz+xz)}{xyz}.
\label{eqn:GZV}
\end{equation}
\end{prop}

\begin{proof}
Define
\begin{equation}
F(a)=x^2(y+z)+y^2(x+z)+z^2(x+y)+(a+3)xyz.
\label{eqn:F}
\end{equation}
Equation (\ref{eqn:Z}) follows immediately from the identity
$$-3xyz+x(y^2+yz+z^2)=x(y-z)^2$$
by cyclically permuting the variables $(x \rightarrow y \rightarrow z \rightarrow x)$.  In order to prove (\ref{eqn:GZV}), we convert to absolute barycentric coordinates by dividing the coordinates of $Z$ by their sum, which is
$$x(y-z)^2+y(x-z)^2+z(x-y)^2=x^2(y+z)+y^2(x+z)+z^2(x+y)-6xyz=F(-9).$$
As above, the sum of the coordinates of $V$ is
$$x^2(y+z)+y^2(x+z)+z^2(x+y)+3xyz = F(0)=(x+y+z)(xy+xz+yz).$$
This shows that $V$ is ordinary whenever $P$ and $P'$ are ordinary.  Putting this into (\ref{eqn:Z}) gives
$$\frac{1}{F(-9)}Z=\frac{-9xyz}{F(-9)}\left(\frac{1}{3},\frac{1}{3},\frac{1}{3}\right)+\frac{F(0)}{F(-9)}\left(\frac{1}{F(0)}V\right).$$
Now $F(0)-9xyz=F(-9)$, so this relation implies that the signed ratio $GZ/ZV$ is given by
$$\frac{GZ}{ZV}=\frac{F(0)/F(-9)}{(-9xyz)/F(-9)}=\frac{-1}{9}\frac{F(0)}{xyz},$$
which agrees with (\ref{eqn:GZV}).  (See \cite{buk}, p. 28.)
\end{proof}

\begin{prop}
If the points $S,V$ are ordinary, we have
$$S = (xyz)G+V \ \ \textrm{and} \ \ \frac{GS}{SV}=\frac{(x+y+z)(xy+xz+yz)}{3xyz}.$$
In particular, the cross ratio $(GV,SZ) = -3$.
\end{prop}

\begin{proof}
The relation betwen $S, G$, and $V$ is immediate from (8) and (\ref{eqn:barV}).  This gives that
$$\frac{1}{F(3)}S=\frac{3xyz}{F(3)}\frac{1}{3}G+\frac{F(0)}{F(3)}\frac{1}{F(0)}V,$$
from which the formula $\frac{GS}{SV}=\frac{F(0)}{3xyz}$ follows.  Then
$$(GV,SZ)=\frac{GS}{GZ} \frac{VZ}{VS} = \frac{GS/SV}{GZ/ZV}=-3.$$ 
\end{proof}

We now determine the homothety ratio of the map $\textsf{M}$.  From \cite{mm3} we have the generalized circumcenter $O=T_{P'}^{-1} \circ K(Q)$ given by
\begin{align}
O &= \left(
\begin{array}{ccc}
 -xx' & x'y  & x'z  \\
 y'x & -yy'  & y'z  \\
 z'x & z'y  &  -zz'
\end{array}
\right)
\left(
\begin{array}{ccc}
 0 & 1  & 1  \\
 1 & 0  & 1  \\
 1 & 1  &  0 
\end{array}
\right) (x(y+z),y(x+z),z(x+y))^t\\
&= (x(y+z)^2x'',y(x+z)^2y'',z(x+y)^2z''),
\end{align}
where
\begin{equation}
x''=xy+xz+yz-x^2, \ \ y''=xy+xz+yz-y^2, \ \ z''= xy+xz+yz-z^2.
\end{equation}
We will use the fact that $\textsf{M}(O)=Q$ to determine the ratio $SQ/SO$ in the next proposition.  Note that the sum of the coordinates of $O$ is
$$x(y+z)^2x''+y(x+z)^2y''+z(x+y)^2z''=8xyz(xy+xz+yz),$$
while the sum of the coordinates of the point $S$ is
$$x(y+z)^2+y(x+z)^2+z(x+y)^2=x^2(y+z)+y^2(x+z)+z^2(x+y)+6xyz=F(3).$$

\begin{prop}
\label{prop:SOQ}
We have that
$$(x+y)(x+z)(y+z)Q=2(xy+xz+yz)S-O$$
and when $P'$ and $S$ are ordinary points, the signed ratio $SQ/SO$ is given by
$$\frac{SQ}{SO}=-\frac{QS}{SO}=\frac{-4xyz}{(x+y)(x+z)(y+z)}.$$
\end{prop}

\begin{proof}
From the computation following the proof of Proposition \ref{prop:fixed} we have that
$$\textsf{M}(O)=8xyz(xy+xz+yz)S-4xyzO,$$
from which we obtain
$$(x+y)(x+z)(y+z)Q=2(xy+xz+yz)S-O.$$
Using the fact that $F(-1)=(x+y)(x+z)(y+z)$, we obtain the relation in absolute barycentric coordinates given by
$$\frac{F(-1)}{F(3)}\frac{1}{2(xy+xz+yz)}Q+\frac{4xyz}{F(3)}\frac{1}{8xyz(xy+xz+yz)}O=\frac{1}{F(3)}S,$$
where $F(-1)+4xyz=F(3)$.  This proves the second assertion.
\end{proof}

\begin{cor}
If $P, P', Z,$ and $S$ are ordinary points, the homothety ratio of the map $\textsf{M}$ is
$$k =\frac{SQ}{SO}=\frac{4}{9\frac{GZ}{ZV}+1} .$$
\label{cor:conj}
\end{cor}

\begin{proof}
This follows immediately from Propositions \ref{prop:ZGV} and \ref{prop:SOQ}, using the fact that
$$-(x+y+z)(xy+xz+yz)+xyz=-(x+y)(x+z)(y+z).$$
\end{proof}

\end{section}

\begin{section}{Elliptic curves over $\mathbb{R}$.}

Let the quantity $a$ be defined by
$$a = 9 \frac{GZ}{ZV}.$$
Then Proposition \ref{prop:ZGV} gives that
\begin{equation}
a=-\frac{(x+y+z)(xy+yz+xz)}{xyz}.
\end{equation}
It follows that the set of ordinary points $P$, for which $Z$ and $V$ are ordinary, and $GZ/ZV=a/9$ is fixed, coincides with the set of $P$ whose coordinates satisfy
$$(x+y+z)(xy+yz+xz)+axyz=0,$$
or
\begin{equation}
E_a: \ \ x^2(y+z)+y^2(x+z)+z^2(x+y)+(a+3)xyz=0.
\label{eqn:E}
\end{equation}
The left side of this equation is exactly the quantity $F(a)$ that we defined in (\ref{eqn:F}).  We note that the set of points, for which $Z$ is infinite, is the set of points for which $F(-9)=0$, and the set of points, for which $V$ is infinite, is the set of points for which $F(0)=(x+y+z)(xy+xz+yz)=0$; the latter is the union of the line at infinity $l_\infty$ and the Steiner circumellipse $\iota(l_\infty)$.  Thus, if $a \neq 0, -9$ is a real number, equation (\ref{eqn:E}) describes the set of ordinary points for which $GZ/ZV=a/9$.  Also, the set of $P$ for which $S$ is an infinite point is the set of $P$ for which $F(3)=0$; this set was studied in the paper \cite{mmv}.  Thus, Corollary \ref{cor:conj} holds for all $a \neq 0, 3, -9$. \medskip

The curve $E_a$ turns out to be an elliptic curve, for $a \neq 0, -1, -9$.  To see this, put $z=1-x-y$ in the equation (\ref{eqn:E}).  This gives the affine equation for $E_a$ in terms of absolute barycentric coordinates $(x, y, 1-x-y)$:
\begin{equation}
E_a: \ \ (ax+1)y^2+(ax+1)(x-1)y+x^2-x=0.
\label{eqn:affE}
\end{equation}
We call this curve the {\it geometric normal form} of an elliptic curve.  The discriminant of the equation (\ref{eqn:affE}) with respect to $y$ is
$$D=(ax+1)^2(x-1)^2-4(ax+1)(x^2-x)=(ax+1)(x-1)(ax^2-(a+3)x-1).$$
This polynomial has discriminant $d=256a^2(a+1)^3(a+9)$ and is therefore square-free in $\mathbb{R}[x]$ if and only if $a \neq 0, -1, -9$.  For these values of $a$, the curve $E_a$ is birationally equivalent to $Y^2=D$, where $D$ is quartic in $x$, a curve which is well-known to be an elliptic curve.  Alternatively, we can compute the partial derivatives
\begin{align*}
\frac{\partial F}{\partial x} &= 2x(y+z)+y^2+z^2+(a+3)yz,\\
\frac{\partial F}{\partial y} &= 2y(x+z)+x^2+z^2+(a+3)xz,\\
\frac{\partial F}{\partial z} &= 2z(x+y)+x^2+y^2+(a+3)xy,
\end{align*}
and check that the equations $\frac{\partial F}{\partial x} = \frac{\partial F}{\partial y}  = \frac{\partial F}{\partial z}  = 0$  have no common solution with $(x,y,z) \neq (0,0,0)$, for $a \notin \{0,-1,-9\}$.  For example, subtracting the first two equations gives
$$2z(x-y)+y^2-x^2+(a+3)z(y-x)=(y-x)((a+1)z+x+y)=0.$$
Hence, $x=y$ or $x+y=-(a+1)z$.  Since the above partials arise from each other by cyclically permuting the variables, we also have that $y=z$ or $y+z=-(a+1)x$, and $ z=x$ or $z+x=-(a+1)y$.  Thus, either: 1) $x=y=z$; or 2) $x=y$ and $z=-(a+2)x$, or similar equations hold resulting from a cyclic permutation; or 3) $x+y=-(a+1)z$ along with the two equations arising from cyclic permutations.  In Case 1, $F(a) = (a+9)x^3 = 0$; in Case 2, $F(a)= a(a+1)x^3=0$; and in Case 3, the determinant of the resulting $3 \times 3$ system is $-a^2(a+3)$.  The first two cases are clearly impossible, so we are left with Case 3, with $a=-3$.  In this case $(x+y-2z)-(x-2y+z)=3y-3z=0$, so $x=y=z$ by symmetry and we are in Case 1 again.  Therefore, $F(a)=0$ is a non-singular cubic curve, which implies it is an elliptic curve, since it has a rational point.
\medskip

\noindent {\bf Remark.}The curve $E_a$ always has a torsion group of order $6$ consisting of rational points.  The points $A=(1,0,0), B=(0,1,0), C=(0,0,1)$ are on the curve, as well as the points $A_\infty = (0,1,-1), B_\infty = (1,0,-1), C_\infty = (1,-1, 0)$, which are the infinite points on the lines $BC, CA$, and $AB$, respectively.  We take the base point of the additive group on $E_a$ to be the point $O=A_\infty=(0,1,-1)$.  Putting $x=0$ in (\ref{eqn:E}) gives $yz(y+z)=0$, so the above 6 points are the only points on the sides of $ABC$ which lie on $E_a$.  See \cite{mmv}. \medskip

We summarize the above discussion as follows.

\begin{thm}
The locus of points $P$, for which the point $Z$ is ordinary and $\frac{GZ}{ZV} = \frac{a}{9} \notin \{0, \frac{-1}{9}, -1\}$, for some fixed $a \in \mathbb{R}$, coincides with the set of points on the elliptic curve $E_a$ defined by (\ref{eqn:E}), minus the points in the set
$$T = \{(1, 0, 0), (0, 1, 0), (0, 0, 1), (0, 1, -1), (1, 0, -1), (1, -1, 0)\}.$$
In particular, the locus of $P$ for which the homothety ratio of the map $\textsf{M}=T_P \circ K^{-1} \circ T_{P'}$ is $k = \frac{4}{a+1}$, for fixed real $a \notin \{3,0,-1,-9\}$, is an elliptic curve minus the points on the sides of triangle $ABC$.
\end{thm}

If $a=3$, then $S \in l_\infty$, and $\textsf{M}$ is a translation.  The elliptic curve $E_3$ was considered in \cite{mmv}.  When $a=-5$, $k=-1$ and $\textsf{M}$ is a half-turn.  The elliptic curve $E_{-5}$ was considered in \cite{mms}.  \medskip

We have the following result for the $j$-invariant of $E_a$.

\begin{prop}
\label{prop;j}
For a real number $a \neq 0, -1, -9$, the $j$-invariant of the elliptic curve $E_a$ is
$$j(E_a) = \frac{(a+3)^3(a^3+9a^2+3a+3)^3}{a^2(a+1)^3(a+9)}.$$
\end{prop}

\begin{proof}
We compute the $j$-invariant of the curve
\begin{equation}
Y^2=(ax+1)(x-1)(ax^2-(a+3)x-1)
\label{eqn:quartic}
\end{equation}
by converting it to a curve in Legendre normal form:
$$E': \ v^2=u(u-1)(u-\lambda),$$
and using the formula
$$j(E') = \frac{2^8(\lambda^2-\lambda+1)^3}{(\lambda^2-\lambda)^2}.$$
Putting $x=\frac{u+1}{u-a}$ and $g(x)=(ax+1)(x-1)(ax^2-(a+3)x-1)$, we have
$$g\left(\frac{u+1}{u-a}\right)=\frac{(a+1)^2}{(u-a)^4} u(-4u^2+(a^2+6a-3)u+4a).$$
Hence, (\ref{eqn:quartic}) is birationally equivalent to
\begin{equation}
Y_1^2=u\left(-4u^2+(a^2+6a-3)u+4a\right).
\label{eqn:Y_1}
\end{equation}
We let $\alpha, \beta$ be the roots of the quadratic in $u$ on the right side of this equation.  Then
$$\alpha, \beta = \frac{a^2+6a-3}{8} \pm \frac{(a+1)}{8} \sqrt{(a+1)(a+9)},$$
and the curve (\ref{eqn:Y_1}) is equivalent over $\mathbb{C}$ to
$$Y_2^2=u(u-\alpha)(u-\beta),$$
which is in turn equivalent over $\mathbb{C}$ to
$$E': \ v^2=u(u-1)\left(u-\frac{\alpha}{\beta}\right).$$
A calculation on Maple with $\lambda=\alpha/\beta$ gives that
$$j(E')=\frac{(a+3)^3(a^3+9a^2+3a+3)^3}{a^2(a+1)^3(a+9)}.$$
This proves the proposition.
\end{proof}

We now prove the following result.

\begin{thm}
Let $E$ be any elliptic curve whose $j$-invariant is a real number.  Then $E$ is isomorphic to the curve $E_a$ over $\mathbb{R}$ for some real value of $a \notin \{0, -1, -9\}$.
\end{thm}

\begin{proof}
Letting $f(x)$ represent the function
$$f(x)=\frac{(x+3)^3(x^3+9x^2+3x+3)^3}{x^2(x+1)^3(x+9)},$$
we just have to check that $f(\mathbb{R}-\{0,-1,-9\})=\mathbb{R}$.  This is a straightforward calculus exercise, which we leave to the reader.  We only note that
$$f'(x) = \frac{6(x+3)^2(x^3+9x^2+3x+3)^2(x^2+6x-3)(x^4+12x^3+30x^2+36x+9)}{x^3(x+1)^4(x+9)^2};$$
that the minimum value of $f(x)$ for $x<-9$ is $1728$; the maximum value of $f(x)$ for $-9<x<-1$ is also $1728$; and that $f(x)$ approaches $-\infty$ as $x$ approaches the asymptotes $x=-9$ (from the right) and $x=-1$ (from the left).  This is enough to prove the assertion.  If $a \in \mathbb{R}$ satisfies $j(E_a) = f(a) = j(E)$, then $E \cong E_a$ over $\mathbb{R}$.
\end{proof}

\noindent {\bf Remark.} The real values of $a$ for which $j(E_a)=1728$ are the real roots of the equation
$$(x^2+6x-3)(x^4+12x^3+30x^2+36x+9)=0,$$
and are given explicitly by
$$a=-3 \pm 2\sqrt{3}, \ -3-\sqrt{3} \pm \sqrt{9+6\sqrt{3}}.$$
\medskip

As an example, the curve $E_{-3}$ has $j(E_{-3})=0$, and affine equation
$$E_{-3}: \ \ (3x-1)y^2+(3x-1)(x-1)y-x^2+x=0.$$
Putting $x=\frac{u}{u-2}$ and $y=\frac{-(3x-1)(x-1)+4v/(u-2)^2}{2(3x-1)}$ yields the isomorphic curve
$$E: \ \ v^2=u^3+1.$$
This curve has exactly $6$ rational points, namely, the points $(2,\pm3),(-1,0),(0,\pm1)$, and the base point $O$.  For which real quadratic fields $K=\mathbb{Q}(\sqrt{d})$ is there a point on $E$ defined over $K$?  For which values of $n \ge 2$ does $E$ have a {\it real} torsion point of order $n$?

\bigskip

\end{section}

\noindent Dept. of Mathematics, Maloney Hall\\
Boston College\\
140 Commonwealth Ave., Chestnut Hill, Massachusetts, 02467-3806\\
{\it e-mail}: igor.minevich@bc.edu
\bigskip

\noindent Dept. of Mathematical Sciences\\
Indiana University - Purdue University at Indianapolis (IUPUI)\\
402 N. Blackford St., Indianapolis, Indiana, 46202\\
{\it e-mail}: pmorton@iupui.edu

\end{document}